\documentclass[11pt,leqno]{amsart}
\usepackage[mathcal]{eucal}
\usepackage[utf8]{inputenc}
\usepackage{amsfonts,amssymb}
\usepackage{hyperref}
\usepackage{graphicx}
\usepackage{enumerate,float}
\usepackage{caption}
\usepackage{xfrac}

\usepackage{tikz}
\tikzset{every picture/.style={line width=0.75pt}} 

\numberwithin{equation}{section}

\theoremstyle{plain}
\newtheorem{lemma}{Lemma}[section]
\newtheorem{proposition}[lemma]{Proposition}
\newtheorem{theorem}[lemma]{Theorem}

\theoremstyle{remark}
\newtheorem{remark}[lemma]{Remark}
\newtheorem*{notation*}{Notation}

\theoremstyle{definition}
\def\U{\mathcal{U}}
\def\R{\mathbb{R}}
\def\Z{\mathbb{Z}}
\def\GG{\mathfrak{S}}
\def\SS{\mathfrak{S}}
\def\La{\mathcal{R}}
\def\L{\mathfrak{L}}

\def\x{\mathbf{x}}
\def\I{\mathfrak{I}}
\def\XX{\mathfrak{X}}
\def\KK{\mathfrak{K}}
\def\MM{\mathfrak{N}}
\def\UU{\mathfrak{U}}
\def\0{\mathbf{0}}
\def\NN{\mathfrak{N}}
\def\N{\mathbb{N}}
\def\K{K}

\DeclareMathOperator{\End}{End}
\DeclareMathOperator{\rk}{rk}
\DeclareMathOperator{\ad}{ad}
\DeclareMathOperator{\Der}{Der}
\DeclareMathOperator{\Ad}{Ad}
\DeclareMathOperator{\charac}{char}
\DeclareMathOperator{\Mat}{M}

\makeatletter
\def\moverlay{\mathpalette\mov@rlay}
\def\mov@rlay#1#2{\leavevmode\vtop{%
   \baselineskip\z@skip \lineskiplimit-\maxdimen
   \ialign{\hfil$\m@th#1##$\hfil\cr#2\crcr}}}
\newcommand{\charfusion}[3][\mathord]{
    #1{\ifx#1\mathop\vphantom{#2}\fi
        \mathpalette\mov@rlay{#2\cr#3}
      }
    \ifx#1\mathop\expandafter\displaylimits\fi}
\makeatother

\makeatletter
\def\author@andify{%
  \nxandlist {\unskip ,\penalty-1 \space\ignorespaces}%
    {\unskip {} \@@and~}%
    {\unskip \penalty-2 \space \@@and~}%
}
\makeatother
\title[Ado's Theorem for PIDs]{A remark on Ado's Theorem for principal ideal domains}

\author[A. Zozaya]{Andoni Zozaya} 
\address{Andoni Zozaya: Department of Mathematics, University
  of the Basque Coun\-try UPV/EHU, 48940 Leioa, Spain}
\email{andoni.zozaya@ehu.eus}
\date{}

\makeatletter
\@namedef{subjclassname@2020}{\textup{2020} Mathematics Subject
  Classification}
\makeatother
\thanks{The author acknowledges support by the Basque Government, project IT483-22, and the Spanish Government, project PID2020-117281GB-I00, partly with ERDF funds.}

\begin{document}

\maketitle

\begin{abstract}
Ado's Theorem had been extended to principal ideal domains independently by Churkin and Weigel. They demonstrated  that if $R$ is a principal ideal domain of characteristic zero and $\L$ is a Lie algebra over $R$ which is also a free $R$-module of finite rank, then $\L$ admits a finite faithful Lie algebra representation over $R$. \\
\smallskip

We present a quantitative proof of this result, providing explicit bounds on the degree of the Lie algebra representations in terms of the rank of the free module. To achieve it, we generalise an established  embedding theorem for complex Lie algebras: any Lie algebra as above embeds within a larger Lie algebra that decomposes as the direct sum of its nilpotent radical and another subalgebra.
\end{abstract}


\section{Introduction}

Ado's Theorem \cite{Ado} states that every finite dimensional Lie algebra $\L$ over a field $K$ of characteristic zero admits a finite faithful Lie algebra representation, that is, there exists a finite dimensional $K$-vector space $V$ and a $K$-Lie algebra monomorphism $\Phi \colon \L \hookrightarrow \End_K(V).$ Obviously, this is equivalent to the existence of a finite matricial representation $\tilde{\Phi} \colon \L \hookrightarrow \Mat_n(K),$ and the integer $n = \dim_K{V}$ is called the degree of the representation.

From most of the proofs of Ado’s Theorem it follows that $\deg \Phi$, the degree of the representation $\Phi,$ is bounded in terms of $\dim_K \L,$ the $K$-vector space dimension of $\L.$ That is to say, if we define the degree of a Lie algebra $\L$ as
\begin{equation}
\label{eq: define-degree}
\deg \L := \min \left\{ \deg \Phi \mid \Phi \textrm{ faithful Lie algebra representation of } \L \right\},
\end{equation}
then the integer $\deg{\L}$ is bounded only in terms of $\dim_K\L$.

Some quantitative bounds are known for $\deg \L,$ specially when $\L$ is nilpotent. For instance,  if $\L$ is a nilpotent Lie algebra of dimension $d$ and nilpotency class $c$,  according to \cite[Corollary to Theorem 4]{Birk},  $\deg \L \leq \frac{d^{c+1} -1}{d-1}$ or according to \cite[Corollary 5.1]{Gra},  $\deg{\L} \leq {{d+c} \choose{c}}.$ Remarkably, if the nilpotency class is fixed, $\deg \L$ is polynomial in $\dim_K \L$. In general, for a nilpotent $d$-dimensional $K$-Lie algebra $\L$, Burde \cite{Burde} proved that 
$$\deg \L \leq \eta \frac{2^d}{\sqrt{d}},$$
where $\eta \sim 2.763$.

Iwasawa \cite{Iwa} extended Ado's Theorem to Lie algebras over fields of positive characteristic, and there are further generalizations in the base ring. Following the terminology of \cite{Wei}, for a general ring $R$ we denote by $R$-Lie lattice an $R$-Lie algebra that is a free $R$-module of finite rank as well. Actually, any $R$-Lie algebra that admits a finite matricial representation is indeed an $R$-Lie lattice. Conversely, suppose that $R$ is a principal ideal domain (PID) of characteristic zero or a general ring of positive characteristic, Churkin \cite{Chur} and Weigel \cite{Wei} proved that every $R$-Lie lattice $\L$  admits a finite faithful $R$-Lie algebra representation $\Phi \colon \L \hookrightarrow \End_R(V),$ where $V$ is a free $R$-module of finite rank. Like for fields, the degree of the preceding representation $\Phi$ is defined to be $\rk_R V,$ the rank of $V$ as a free $R$-module,  and the degree of an $R$-Lie lattice is defined exactly as in (\ref{eq: define-degree}). 

Suppose that $R$ is a PID of characteristic zero. Both \cite{Chur} and \cite{Wei} follow Jacobson’s proof of the Theorem of Ado (see \cite[Chapter VI]{Jacobson}) ---which in its turn, is based on a proof due by Harish-Chandra \cite{HC}---, but, unlike for fields, it cannot be directly affirmed that $\deg \L$ depends uniquely on $\rk_R \L.$ In fact, in  \cite[Proposition 3.4]{Wei}, the degree-to-be is finite because $R$ is a Noetherian ring, and so a particular ascending chain of ideals must be stationary. However, the length of the chain ---which eventually will be the degree of the representation---, might not be bounded in terms of $\rk_R \L$.

In this note, we collect several existing proofs of Ado’s Theorem, and by adapting them to PIDs we prove the following:

\begin{theorem}
    \label{thm: Ado}
    Let $R$ be a PID of characteristic zero and let $\L$ be an $R$-Lie lattice of rank $r.$ Then,
    $\deg \L \leq r +  \eta \frac{2^r}{\sqrt{r}},$
where $\eta \sim 2.763.$
\end{theorem}

In particular, we recover for PIDs the best bound yet known over fields of characteristic zero. \\

More concretely, in Subsections \ref{subsec: nilpotent} and \ref{subsec: splittable} we reproduce quantitative results about the representability of  nilpotent and splittable $R$-Lie lattices, and in Subsection \ref{subsec: embedding} (see Theorem \ref{thm: embedding}) we prove the following:

\begin{theorem} \label{thm: embedding PID} Let $R$ be a PID of characteristic zero and let $\L$ be an $R$-Lie lattice. There exists an $R$-Lie lattice of the form $\bar{\L} = R_n(\bar{\L}) \rtimes \SS$ extending $\L,$ where $R_n(\bar{\L})$ is  the nilpotent radical of $\bar{\L}.$
\end{theorem}

This result is based on the analogue for complex Lie algebras proved by Neretin \cite{Neretin}, and previously discussed in \cite{Malcev, Reed}. Lastly, Theorem \ref{thm: Ado}  is proved in Subsection \ref{subsec: Ado} using Theorem \ref{thm: embedding PID} and the arguments of the previous subsections. \\

Finally, we must note that for rings of positive characteristic, the generalisation is proved reproducing word-by-word the original demonstration of Iwasawa, and therefore we obtain the same bound we had for these fields, namely 
$$\deg \L \leq n^{\rk^3 \L},$$
where $n = \charac{R}$ (compare with \cite[Section 6.24]{Bah}).

\smallskip

\textbf{Notation} Hereinafter $R$ will always be a PID of characteristic zero, and we will use $K$ to denote fields. For an $R$-Lie lattice $\L,$ $R_n(\L)$ and $R_s(\L)$ refer to the nilpotent and solvable radicals of $\L.$ We denote by $\dim_K$ the $K$-vector space dimension, by $\rk_R$ ($\rk$ when $R$ is clear from the context) the rank of a free $R$-module, by $\langle X \rangle_R$ the $R$-module generated by a set $X,$ and  $\I \leq \L$ and $\I \unlhd \L$ represent respectively that $\I$ is a subalgebra and an ideal of $\L.$ We will use the abbreviation $[\I_1, \dots, \I_n] = [[\I_1, \dots, \I_{n-1}], \I_n]$ for iterated Lie brackets, and throughout the manuscript ``$:=$" is used to mean \textit{defined to be} in contrast with ''$=$'', which is used to denote \textit{equal to}. 

Finally, we recall that an $R$-submodule $\I \leq \L$ is isolated if whenever $r x \in \I$ for some $r \in R$ and $x \in \L,$ then $x \in \I,$ that is, the quotient $R$-module $\sfrac{\L}{\I}$ is torsion-free, and thus free.


\section{Preliminaries: adjoint and regular representations}

There are two natural Lie algebra representations in any $R$-Lie lattice $\L.$ On the one hand, by virtue of Jacobi's identity the adjoint representation $\Ad \colon \L \rightarrow \End_R(\L),$ $x \mapsto \ad_x,$ where $\ad_x \colon \L \rightarrow \L,$ $y \mapsto [x,y],$ is a finite Lie algebra representation. However, this representation is not faithful as its kernel is the centre of $\L,$ namely 
$$Z(\L) = \{ x \in \L \mid [x, y] = 0 \ \forall y \in \L \}.$$
In particular, when $\L$ is a semisimple $R$-Lie lattice, i.e. $\L$ has no abelian ideal,  then $\deg \L \leq \rk_R \L.$ 

Typically, Ado's Theorem is proved by constructing a \textit{finite} representation $\Phi \colon \L \rightarrow \End_R(W)$ that is faithful in $Z(\L),$ and then taking the finite faithful representation $\Ad \oplus \Phi.$ \\

\smallskip 

On the other hand, $\L$ acts on its universal enveloping algebra $\U_R(\L).$ Indeed, the tensor algebra of $\L$ is 
$$\mathbf{T}_R(\L) = R \oplus \L_1 \oplus \L_2 \oplus \dots \oplus \L_i \oplus \dots,$$
where $\L_i := \L \otimes \stackrel{(i)}{\dots} \otimes \L$ is an $R$-module with the natural $R$-module structure of the tensor product, and the multiplication in $\mathbf{T}_R(\L)$ is defined extending by linearity the rule
$$(x_1 \otimes \dots \otimes x_i) \otimes (y_1 \otimes \dots \otimes y_j) = x_1 \otimes \dots \otimes x_i \otimes y_1 \otimes \dots \otimes y_j.$$

Then, the universal enveloping algebra of $\L$ is
$$\U_R(\L) : = \frac{\mathbf{T}_R(\L)}{\mathfrak{R}},$$
where $\mathfrak{R}$ is the ideal generated by the elements 
\begin{equation}
\label{eq: define the ideal}
[x,y] - (x \otimes y - y \otimes x), \ \forall x, y \in \L.
\end{equation}
The image of $x_{i_1} \otimes \dots \otimes x_{i_t}$ in $\U_R(\L)$ will be simply denoted by the monomial $x_{i_1} \dots x_{i_t}.$\\

Since $R$ is a PID and $\L$ is finitely generated, the natural inclusion $\iota \colon \L \cong \L_1 \hookrightarrow \U_R(\L)$ is a monomorphism (see \cite[Theorem 3.2]{Wei}), and therefore, we can assume that $\L \subseteq U_R(\L).$ The universal enveloping algebra is characterised by the Poincar\'e-Birkhoff-Witt Theorem:

\begin{theorem}[\textup{cf. \cite[Theorem 3.2]{Wei}}]
\label{PBW}
Let $\L$ be an $R$-Lie lattice with basis $\{x_1, \dots, x_r\}.$ Then $\U_R(\L)$ is a free $R$-module with basis 
\begin{equation}
    \label{eq: basis}
\left\{ x_1^{\alpha_1}\dots x_r^{\alpha_r}  \mid \alpha_i \in \N_0 \right\},
\end{equation}
where $x_1^0\dots x_r^0 =1$ is the identity of $R$.
\end{theorem}

The idea is that given two monomials, their product can expressed as a suitable linear combination of elements of the form (\ref{eq: basis}) by successively applying the identity $x_jx_i = x_ix_j - [x_i, x_j]$ to reorder the indeterminates. 

As we have said, $\L$ acts on $\U_R(\L)$ by multiplication and this gives rise to the (left) regular representation $\La \colon \L \hookrightarrow \End_R\left(\U_R(\L)\right),$ where $\La(x)$ is the left multiplication map $\ell_x \colon \U_R(\L) \rightarrow \U_R(\L),$ $y \mapsto xy.$ 

By virtue of (\ref{eq: define the ideal}), $\La$ is a Lie algebra representation, and it is faithful as 
$$\ell_x(1) = x \neq y = \ell_y(1) \textrm{ for all } x,y \in \L.$$ 
Although $\La$ is not finite, by virtue of the universal property of the enveloping algebra (see \cite[Proposition 3.1]{Wei}), every finite Lie algebra representation of $\L$ factors through $\U_R(\L).$ Hence, $\L$ admits a finite faithfull Lie algebra representation if and only if there exists an isolated ideal 
$\XX \unlhd \U_R(\L)$ such that $\XX \cap \L = \{0 \}.$ Actually, for the \textit{if} it is enough to consider the induced action on the free $R$-module $\sfrac{\U_R(\L)}{\XX}.$ 


\section{Main result}

Let $\L$ be an $R$-Lie lattice and let $K$ be a field extending $R$ --e.g. the fraction field of $R$--, the tensorial $K$-Lie algebra $\L_K := \L \otimes_R \K$ will be useful in the following subsections. Note in passing that even though $\L_K$ admits a matricial representation $\Phi \colon \L_K \hookrightarrow \Mat_n(K),$ $\Phi|_{\L}$ might not be a  matricial representation over $R.$ 

\subsection{Nilpotent Lie lattices}

\label{subsec: nilpotent}

For nilpotent $R$-Lie lattices the construction of Birkhoff \cite{Birk} is still valid over PIDs.\\

Suppose that $\L$ is a nilpotent $R$-Lie lattice of nilpotency class $c,$ and let $K$ be the fraction field of $R.$ Since the Lie bracket is bilinear, $\L_K = \L \otimes_R K$ is also a nilpotent Lie algebra of nilpotency class $c.$ 

For each $i \in \{1, \dots, c\},$ define the isolated ideal  $\L_i = [\L_K, \stackrel{(i)}{\dots}, \L_K ] \cap \L \unlhd \L,$ and choose a basis $\{x_1, \dots, x_{r}\}$ for $\L$ as free $R$-module in such way that the first elements $x_1, \dots, x_{r_{1}}$ are an $R$-basis for $\L_{c},$ the first elements $x_1, \dots, x_{r_2}$ ($r_2 > r_1$) are an $R$-basis for $\L_{c-1}$ and so forth. In view of Theorem \ref{PBW}, the elements of $\U_R(\L)$ are of the form $\sum_{\alpha \in \N_0^{(r)}} c_\alpha \x^\alpha,$ where $\x^\alpha$ stands for $x_1^{\alpha_1} \dots x_r^{\alpha_r}.$  Accordingly, define a weight function $\omega \colon \U_R(\L) \rightarrow \N_0 \cup \{ \infty \}$ in the following fashion:\\

\begin{center}
\begin{tabular}{ lll} 
\rule{0pt}{2ex}   $\omega(x_i) = \max\left\{ m \mid x_i \in \L_{m} \right\},$ & $\omega(\x^\alpha) = \sum_{i=1}^r \alpha_i \omega(x_i), $ \\ 
\rule{0pt}{3ex}   $\omega\left(\sum_\alpha c_\alpha \x^\alpha \right) = \min\left\{ \omega(\x^\alpha) \mid c_\alpha \neq 0 \right\},$ & $\omega(0) = \infty.$ 
\end{tabular}
\end{center}
\vspace{1em}
 
Observe that $\omega([x_i, x_j]) \geq \omega(x_i) + \omega(x_j) $ for all $i, j \in \{ 1, \dots, r\},$ and so
\begin{equation}
    \label{eq: biderketa}
    \omega(u v ) \geq \omega(u) + \omega(v) \ \forall u, v \in \U_R(\L).
\end{equation}

For each $m \in \N_0,$  consider the isolated $R$-modules
$$
\UU^m(\L) := \left\{ u \in \U_R(\L) \mid \omega(u) > m \right\}
$$
--or simply $\UU^m$ when the lattice is clear from the context--. By (\ref{eq: biderketa}), $\UU^m(\L)$ is an ideal and thus for every $x \in \L$ we have that $\ell_x(\UU^m) \subseteq \UU^m,$ so for any $m \in \N$ the regular representation induces the finite representation
$$
\La_m \colon \L \rightarrow \End_R\left( \frac{\U_R(\L) }{ \UU^m(\L)}\right), \ x \mapsto \ell_x 
,$$
whose kernel is $\L \cap \UU^m(\L)$ 
--with an abuse of notation, whenever $f \in \End_R(\U_R(\L))$ satisfies $f(\XX) \subseteq \XX$ for some ideal $\XX \unlhd \U_R(\L),$ we will keep $f$ to denote the endomorphism of $\End_R\left(\sfrac{\U_R(\L)}{\XX}\right)$ defined as $x + \XX \mapsto f(x) + \XX$--.\\

Since $\L \cap \UU^{c}(\L) =\{0\},$ $\La_c$ is a finite faithful representation and its degree is
$$ \left|\{\x^\alpha \mid \omega(\x^\alpha) \leq c \} \right|,$$
as these monomials are a basis for $\sfrac{\U_R(\L)}{\UU^c}.$ Finally, this number was bounded by Burde (see \cite[Lemma 5(3) and Proposition 6]{Burde}):
\begin{equation}
\label{eq: bornea}
\deg \La_c  =  \rk_R \left( \frac{\U_R(\L)}{\UU^c(\L)} \right) = |\{\x^\alpha \mid \omega(\x^\alpha) \leq c \}| \leq \eta \frac{2^r}{\sqrt{r}},
\end{equation}
where $\eta = \sqrt{\frac{2}{\pi}} \prod_{l=1}^\infty \frac{2^l}{2^l-1} \sim 2.763.$


\subsection{Splittable Lie lattices}
\label{subsec: splittable}

We say that $\L$ is splittable if the short exact sequence 
$$0 \rightarrow R_n(\L) \rightarrow \sfrac{\L}{R_n(\L)} \rightarrow 0$$
splits in the category of $R$-Lie algebras, that is, if there exists an $R$-Lie subalgebra $\GG \leq \L$ such that $\L = R_n(\L) \rtimes \GG.$

In the splittable case we can blend the preceding regular representation for $R_n(\L)$ and representations induced from derivations, namely endomorphisms $D \in \End_R(\L)$ that satisfy Leibniz identity, i.e.
$$D([x,y]) = [x,D(y)] + [D(x), y] \ \forall x, y \in \L.$$
The collection of all derivations of  $\L$ is denoted by $\Der_R(\L)$. For example, in view of Jacobi's identity, $\ad_x$ is a derivation for every $x \in \L.$  Starting from $D \in \Der_R(\L)$  we can define a derivation $D^*$ of $\U_R(\L)$ by imposing Jacobi's identity, that is, by taking the linear extension of the rule
$$D^*(x_{i_1} \dots x_{i_t}) = \sum_{j} x_{i_1} \dots x_{i_{j-1}} D(x_{i_j}) x_{i_{j+1}} \dots x_{i_t},$$
together with $D^*(1)= 0$ as it must happen for every derivation of an algebra with identity. 

In keeping with the notation of the previous subsection, we have:

\begin{lemma}
\label{lem: deribazioak}
Let $\L$ be a nilpotent $R$-Lie lattice and $D \in \Der_R(\L).$ Then $D^*(\UU^m(\L)) \subseteq \UU^m(\L)$ for every $m \in \N.$ 
\end{lemma}
\begin{proof}
Let $c$ be the nilpotency class of $\L$, and let $\{x_1, \dots, x_r\}$ be the basis of $\L$ with respect to which the weight function $\omega$ has been defined. Since $D$ is a derivation, $D(\L_i) \subseteq \L_i$ for all $i \in \{1, \dots, c\},$ so $\omega(D(x_i)) \geq \omega(x_i)$ for all $i \in \{1, \dots, r\}.$ Hence, if $x_{i_1} \dots x_{i_t} \in \UU^m(\L),$ by (\ref{eq: biderketa}),
\[ \omega\left(D^*(x_{i_1} \dots x_{i_t}) \right) \geq \min_{j = 1, \dots, t}\left\{ \omega(x_{i_1} \dots D(x_{i_j}) \dots x_{i_t}) \right\} \geq \omega(x_{i_1} \dots x_{i_t}) > m. \qedhere \]
\end{proof}

\begin{proposition}[Zassenhaus extension\index{Zassenhaus!extension}, cf. \textup{\cite[Chapter VI.2, Theorem 1]{Jacobson}}]
\label{prop: splittable}
Let $\L $ be a splittable $R$-Lie lattice and let $c$ be the nilpotency class of $R_n(\L)$. Then, there exists a finite $R$-Lie algebra representation 
$$\Phi \colon \L \rightarrow \End_R\left( \frac{\U_R(R_n(\L))}{\UU^c(R_n(\L))} \right)$$
 that is injective in $R_n(\L)$ and such that
\begin{equation}
\label{eq: borne nilpotent}
 \deg\Phi \leq \eta \frac{2^{\rk{R_n(\L)}}}{\sqrt{\rk R_n(\L)}}, 
\end{equation}
where $\eta \sim 2.763.$
\end{proposition}

 \begin{proof}
Denote $R_n(\L)$ by $\NN,$ then $\L = \NN \rtimes \SS$ for some $R$-Lie subalgebra $\SS \leq \L$. By Lemma \ref{lem: deribazioak}, $\ad_{x}^*\left(\UU^c(\NN) \right) \subseteq \UU^c(\NN)$ for all $x \in \L,$ so we can define the map
$$\Phi \colon \L= \NN \oplus \SS \rightarrow \End_R\left(\sfrac{ \U_R(\NN) }{\UU^c(\NN)}\right), \ n  + s \mapsto  \ell_n + \ad_{s}^* .$$
In order to show that it is an $R$-Lie algebra homomorphism, it suffices to confirm that 
$$\Phi\left(  [s, n ]\right) = \left[\Phi(s), \Phi(n) \right] = \left[\ad_{s}^*, \ell_n \right]$$
for all $n \in \NN$ and $s \in \SS.$ Indeed, for any $n \in \NN$ and  $D \in \Der_R\left(\U_R(\NN)\right):$
$$[D, \ell_n](u) = D \circ \ell_n(u)- \ell_n \circ D(u) = D(n) u = \ell_{D(n)}(u)  \ \forall u \in \NN,$$
and, since $\NN$ is an ideal, $[s, n] \in \NN,$ so
$$\Phi\left([s, n] \right) = \ell_{[s, n]} = \ell_{\ad_{s}^*(n)}= \left[\ad_{s}^*, \ell_n \right] = \left[\Phi(s), \Phi(n) \right].$$
Consequently, $\Phi$ is a finite $R$-Lie algebra representation. In addition,  $\Phi|_\NN$ is nothing but the  faithful representation $\La_c$ of $\NN.$ Finally, (\ref{eq: borne nilpotent}) follows from (\ref{eq: bornea}).
\end{proof}


\subsection{Embedding Theorem} \label{subsec: embedding} 

In  \cite[Chapter IV.2]{Jacobson}, and the succeeding works following it, Levi's Theorem is crucial; namely, if $K$ is a field of characteristic zero there exists a semisimple Lie algebra $\GG \leq \L$ such that $\L = R_s(\L) \rtimes \SS$ (see \cite[Chapter III.9]{Jacobson}). However, this results is not longer true for PIDs. For instance, the $\Z$-Lie algebra $\mathfrak{sl}_2(2\Z) \oplus \mathfrak{t}_2(2\Z)$ ---the direct sum of $2 \times 2$ matrices of trace zero and $2 \times 2$ upper triangular matrices with coefficients in $2\Z$--- does not admit such a decomposition (see \cite[Example in pg. 838]{Chur}).\\

Nevertheless, every Lie lattice embeds in a splittable (in the sense of Subsection \ref{subsec: splittable}) $R$-Lie lattice. In effect, over algebraically closed fields this result was first proved for solvable Lie algebras by Mal'cev \cite {Malcev} and Reed \cite{Reed}, and using similar ideas Neretin \cite{Neretin} proved the following (albeit \cite{Neretin} is about complex Lie algebras, the proof is still valid, with small remarks, for any field of characteristic zero):

\begin{theorem}[cf. \textup{\cite[Lemma 1]{Neretin}}] \label{thm: Neretin} Let $K$ be a field of characteristic zero and $\L$ a finite dimensional $K$-Lie algebra. There exists a  splittable $K$-Lie algebra $  \bar\L = R_n(\bar \L) \rtimes \SS,$ where $\SS$ is reductive, extending $\L.$   
\end{theorem}
\smallskip
The above theorem is proved by successively applying elementary expansions. Indeed, suppose that we have a $K$-Lie algebra $\KK = \MM \rtimes \SS$ extending $\L$ such that $\MM$ is a solvable ideal containing $R_n(\KK)$ and $\SS$ is a reductive --direct sum of a semisimple and an abelian algebra-- subalgebra that acts fully irreducibly on $\MM.$ We shall construct another Lie algebra $\KK'$ extending $\KK$ that satisfies those same conditions.

By \cite[Chapter III.7, Theorem 13]{Jacobson}, $[\MM, \KK] \leq R_n(\KK).$ Thus, unless $\MM$ is nilpotent, there exists an ideal $\I \unlhd \MM$ of codimension one containing $R_n(\KK)$. Since the action of $\SS$ is fully irreducible, there exists an element $y \in \MM \setminus R_n(\KK)$ such that $\MM = \I \oplus Ky$ as $\SS$-modules, in particular, $[y, \SS] \subseteq R_n(\KK) \cap Ky = \{ 0\}.$ Moreover, according to the Jordan-Chevalley decomposition (see \cite[Chapter III.11, Theorem 16]{Jacobson}), the derivation  $\ad_{y} \in \Der_R(\KK)$ decomposes as $d_{s,y} + d_{n,y}$ where  $d_{s,y}$ and $d_{n,y}$ are respectively a semisimple and a nilpotent $K$-linear endomorphism.

\begin{remark} \label{rmk: are-derivations}
Both $d_{s,y}$ and $d_{n,y}$ are in $ \Der_K(\KK)$. Indeed, when  $K$ is algebraically closed it was proved in  \cite[Proposition 3]{Reed}, as 
$d_{s,n}(v) = \alpha v$ provided that $v$ belongs to the generalised $\alpha$-eigenspace of $\ad_y.$

In general, let us write $S = d_{s,y}$ and $N=d_{n,y}$ and let $\bar{K}$ be the algebraic closure of $K.$ Suppose that $\bar{S} + \bar{N}$ is the Jordan-Chevalley decomposition of $\ad_y$ in $\KK_{\bar K} = \KK \otimes_K \bar{K}.$ Then 
$$S \otimes \bar{K} + N \otimes \bar K = \ad_y = \bar{S} + \bar{N}$$
are two decompositions of $\ad_y \in \End_{\bar{K}} (\KK_{\bar{K}}),$ so by the uniqueness $\bar{S} = S \otimes \bar{K}$ and $\bar{N} = N \otimes \bar{K}.$ Finally, since $\bar{S}$ and $\bar{N}$ satisfy Leibniz identity, so do $S$ and $N$. 
\end{remark}

Thus, we can construct a so-called elementary expansion, namely the $K$-Lie algebra 
$$\KK' = \I \oplus \SS  \oplus  Kx' \oplus  K z',$$
where $x'$ and $z'$ are formal symbols satisfying
\vspace{1em}
\begin{center}
\begin{tabular}{ lll} 
\rule{0pt}{2ex}   $[x', u] = d_{n,y}(u),$ & $[z', u] = d_{s,y}(u),$ & $\left[x', z' \right] =0$
\end{tabular}
\end{center}
\vspace{1em}
for every $u \in \I \oplus \SS,$ and where we keep the original Lie bracket for the elements of $\I \oplus \SS.$ Observe that $\KK= \I \oplus Ky\oplus \SS$ embeds as a Lie algebra in $ \KK'$ with respect to  $y = x' +z'.$



In addition, $\ker{\ad_y} \subseteq \ker d_{s,y}$  (see Remark \ref{rmk: are-derivations}), so $[z', \SS] =0$  and $\SS' : = \SS \oplus Kz'$ is a reductive Lie algebra. Moreover, $R_n(\KK) \oplus Kx'$ is the nilpotent radical of $\KK'$,  $\MM':= \I \oplus Kx'$ is solvable, and, since $d_{s,y}$ is a semisimple operator, the action of $\SS'$ in $\MM'$ is fully reducible. In particular, $\KK' = \MM' \rtimes \SS'.$ In passing, note that 
\begin{equation}
\label{eq: decrease dimension}
\dim_K \MM' = \dim_K\MM \textup{ and }
 \dim_K R_n(\KK') = \dim_K R_n(\KK) +1.
\end{equation}

Levi's Theorem gives us the first of the step of the above-described procedure. Indeed, $\L = R_s(\L) \rtimes \SS$ for a semisimple subalgebra $\SS \leq \L,$ and by virtue of Weyl's Theorem on complete reducibility (see \cite[Chapter III.7, Theorem 8]{Jacobson}), the action of $\SS$ on $R_s(\L)$ is fully reducible. Fix bases $\{x_1, \dots,  x_s\}$ of
$R_n(\L)$ and $\{z_1, \dots, z_t\}$ of $\SS.$ In view of (\ref{eq: decrease dimension}), repeating the previous process eventually we obtain a $K$-Lie algebra 
\begin{equation}
\label{eq: hedadura forma}
\bar{\L} = \bar{\NN} \rtimes \bar{\SS} = \langle x_1, \dots, x_s, x_1', \dots, x_r' \rangle_K \rtimes \langle z_1, \dots, z_t, z_1', \dots, z_r' \rangle_K,
\end{equation}
where $\bar{\NN}$ is a nilpotent ideal and $\bar{\SS} \leq \bar{\L}$ is a reductive subalgebra, and a $K$-basis $\{x_1, \dots, x_s, y_1, \dots, y_r, z_1, \dots, z_t\}$ of $\L$ such that $R_s(\L) = \langle x_1, \dots, x_s, y_1, \dots, y_r \rangle_K$ and $y_i = x_i' + z_i'.$   In particular, $\bar \L$ extends $\L,$ and by construction:
\begin{enumerate}[(N1)]
\item \label{item:1} for all $i, j \in \{1,\dots, r\}$ and $ k \in \{1, \dots, t\}$
$$\left[z_i', z_j'\right] = \left[z_i', z_k\right] =0,$$
and therefore  
$$[x_i', z_k] = [y_i, z_k] \in R_n(\L)$$
(see \cite[Chapter II.7, Theorem 13]{Jacobson});

\item \label{item:2} for all $i \in \{1, \dots, s\}$ and $j, k \in \{1, \dots ,r\},$ by \cite[Chapter III.6, Theorem 7]{Jacobson}, $d_{n,y_j}(R_s(\L)) \subseteq R_n(\L),$ so
$$\left[x_j', x_i \right] = d_{n,y_j}(x_i) \in R_n(\L) \textup{ and } \left[x_j', y_k \right] = d_{n,y_j}(y_k)  \in R_n(\L).$$
In particular, $R_n(\L) \unlhd R_n(\bar \L);$ 

\item \label{item:3}$\dim_K R_n(\bar \L) = s + r = \dim_K R_s(\L).$ \label{N4}
\end{enumerate}
\smallskip
Furthermore, in view of (N\ref{item:1})-(N\ref{item:2}), we have that 
\begin{equation}
\label{eq: derivations}
\left[x_i', \L \right] : =  \left\{\left[x_i', u \right] \mid u \in \L \right\} \subseteq R_n(\L) = \langle x_1, \dots, x_s \rangle_K 
\end{equation}
for all $i \in \{1, \dots, r\}.$
\smallskip
As a consequence, we can prove the following strengthened version of Theorem \ref{thm: embedding PID}:
\smallskip
\begin{theorem}
\label{thm: embedding}
Let $R$ be a PID of characteristic zero and let $\L$ be an $R$-Lie lattice. Then $\L$ embeds into a  splittable $R$-Lie lattice $\bar\L$ such that 
\begin{itemize}
    \item[\textup{(i)}] $R_n(\L) \leq R_n(\bar{\L})$ and 
    \item[\textup{(ii)}] $\rk R_n(\bar{\L}) = \rk R_s(\L).$ 
\end{itemize}
\end{theorem}

\begin{proof}
Let $K$ be the fraction field of $R$ and $\L_{\K} : = \L \otimes_R \K.$  According to Theorem \ref{thm: Neretin}, there exists a finite dimensional splittable $\K$-Lie algebra $\L_{\K}' = R_n(\L_{\K}') \rtimes \SS_{\K}'$ extending $\L_{\K}$ and satisfying conditions (N\ref{item:1})-(N\ref{item:3}). Denote for simplicity $\NN_{\K}' : = R_n(\L_{\K}').$ 

Let $\{x_1, \dots, x_s\}$ be a basis for $R_n(\L)$ as free $R$-module, then $R_n(\L_{\K}) = \langle x_1, \dots, x_s\rangle_K$ and, by (\ref{eq: hedadura forma}), there is a $K$-vector space basis of $\NN_K'$ of the form  
$$\left\{x_1, \dots, x_s, x_1', \dots, x_r'\right\},$$
where $s+r = \dim_K R_s(\L_K) = \rk_R R_s(\L)$ (compare with (N\ref{item:3})).\\
\smallskip

Furthermore, by (N\ref{item:2}), there exists $\mu \in R \setminus \{ 0 \}$ such that 
$$\NN : = \langle x_1, \dots, x_s, \mu x_1', \dots, \mu x_r' \rangle_R$$
is a nilpotent $R$-Lie lattice of rank $s +r$, $R_n(\L) = \langle x_1, \dots, x_s \rangle_R \unlhd \NN$ and
$$[\mu x_j', \L] \subseteq \langle x_1, \dots, x_s \rangle_R,$$
for all $j \in \{1, \dots, r\}$ (using (\ref{eq: derivations}) and that $\L$ is finitely generated for the last condition). In particular,
\begin{equation}
\label{eq: are ideals}
[\NN, \stackrel{(i)}{\dots}, \NN, \L ] \leq [\NN, \stackrel{(i)}{\dots}, \NN ] \ \forall i \in \N.
\end{equation}

\smallskip 

Let $\bar{\SS}$ be the projection of $\L$ into $\SS_{\K}',$ that is,
$$\bar{\SS} = \left\{ \sigma \in \SS_{\K}' \mid \exists \ x \in \L, \exists \ n \in \NN_{\K}' \textup { such that } x = n +\sigma \right\}.$$
Then $\bar{\SS}$ is an $R$-Lie algebra.  Indeed, if $x_1 = n_1 + \sigma_1$ and $ x_2 = n_2 + \sigma_2 \in \L,$ where $n_i \in \NN_{\K}'$ and $\sigma_i \in \SS_{\K}'$ ($i \in \{1,2\}$), then
$$ [x_1, x_2] = [n_1, x_2] + [\sigma_1, n_2] + [\sigma_1, \sigma_2],$$
$[x_1, x_2] \in \L,$ $[n_1, x_2] + [\sigma_1, n_2] \in \NN_{\K}'$ and $[\sigma_1, \sigma_2] \in \SS_{\K}'.$ In addition, since $\L$ is finitely generated, $\bar\SS$ is a free $R$-module of finite rank. \\

Moreover, since $\L$ is a finitely generated $R$-module there exists $\lambda \in R \setminus \{0\}$ such that 
\begin{equation}
\label{eq: idazkera}
\L = \frac{1}{\lambda} \NN \oplus \bar\SS.
\end{equation}

Let $c$ be the nilpotency class of $\NN,$ define $\NN_i := [\NN, \stackrel{(i)}{\dots}, \NN],$ for each $i \in \{1, \dots, c\},$ and
$$\bar{\NN} : = \sum_{i=1}^c \frac{1}{\lambda^i} \NN_i \leq \NN_K',$$
 which is a free $R$-module of rank $s + r = \rk R_s(\L)$.

On the one hand, 
$$\left[ \frac{1}{\lambda^i} \NN_i, \frac{1}{\lambda^j} \NN_j \right] = \frac{1}{\lambda^{i+j}} \left[ \NN_i, \NN_j \right] \leq \frac{1}{\lambda^{i +j}} \NN_{i+j},$$
so $\bar{\NN}$ is a nilpotent $R$-Lie lattice. 

On the other hand, by (\ref{eq: idazkera}) and (\ref{eq: are ideals}), 
\begin{align*} 
\left[ \frac{1}{\lambda^i} \NN_i,  \bar{\SS} \right] & \leq \left[ \frac{1}{\lambda^i} \NN_i,  \L + \frac{1}{\lambda}\NN \right]  \leq \frac{1}{\lambda^i}\left[ \NN_i, \L \right] + \frac{1}{\lambda^{i+1}}\left[\NN_i, \NN \right] \\& \leq \frac{1}{\lambda^i} \NN_i + \frac{1}{\lambda^{i+1}} \NN_{i+1} \leq \bar{\NN}
\end{align*}
for every $i \in \{1, \dots, c\}.$ 

Hence, $\bar{\L} := \bar{\NN} \rtimes \bar{\SS}$ is an $R$-Lie lattice that extends $\L;$ by construction $\bar{\L}$ is splittable, $R_n(\L) \leq \bar{\NN} = R_n(\bar \L)$   and  $\rk R_n(\bar \L) = \rk \bar{\NN} =  s +r = \rk R_s(\L).$  
\end{proof}


\subsection{Ado's Theorem for PIDs}
\label{subsec: Ado}

Finally, we gather all the ingredients:

\begin{proof}[proof of Theorem \ref{thm: Ado}]
Let $\L$ be an $R$-Lie lattice of rank $r$. According to Theorem \ref{thm: embedding}, there exists a splittable $R$-Lie lattice $\bar{\L} = R_n(\bar \L) \rtimes \SS$ extending $\L$ such that $R_n(\L)\leq R_n(\bar \L) $ and $\rk R_n(\bar \L) = \rk R_s(\L).$ By Proposition \ref{prop: splittable}, there exists an $R$-Lie algebra representation $\Phi$ of $\bar{\L}$ which is injective in $R_n(\bar \L)$ and whose degree is bounded by $f(\rk{R_n(\bar{\L})}),$ for $f \colon \N_{\geq 1} \rightarrow \R,$ $r \mapsto \eta \frac{2^r}{\sqrt{r}}.$ 

Therefore $\tilde{\Phi}:=\Phi |_\L \oplus \Ad$ is an $R$-Lie algebra representation of $\L$ that is faithful, as
$$\ker\tilde{\Phi} = \ker{\Phi|_\L} \cap \ker{\Ad} \subseteq \left(\L \setminus R_n(\L)\right) \cap Z(\L) = \{0\}.$$
Thus, since $\rk{R_n\left(\bar{\L} \right)} = \rk{R_s(\L)}$ and $f$ is non-decreasing
\[ \deg{\L} \leq \deg{\tilde{\Phi}} = \deg \Phi + \deg{\Ad} \leq  f(\rk{R_s(\L)}) + r \leq f(r) +r. \qedhere \]
\end{proof}

\begin{remark}
For a $K$-Lie algebra $\L$, with $K$ being a field of characteristic zero,  
Harish-Chandra \cite {HC} improved the original result of Ado by constructing a finite faithful representation $\Psi \colon \L \hookrightarrow \End_K(V)$ with the additional property of been a so-called nil-representation, i.e.  $\Psi(x)$ is a nilpotent endomorphism  for every $x \in R_n(\L).$

Note that the representation $\tilde{\Phi}$ of the preceding proof is also a nil-representation, as both $\Phi$ and $\Ad$ are so.
\end{remark}

\end{document}